\documentclass[11pt,reqno,sumlimits]{amsart}
\usepackage{amsfonts, amsmath, amscd, amssymb, euscript, amsthm, array, booktabs,
dcolumn, shortvrb, tabularx, units, url, mathrsfs}
\usepackage[pdftex]{graphicx}
\usepackage[pdftex]{hyperref}
\usepackage[all]{xy}
\normalsize

\DeclareMathOperator{\FL}{FL}
\DeclareMathOperator{\TP}{TP}

\DeclareMathOperator{\rk}{rank}

\DeclareMathOperator{\SSSS}{SS}

\DeclareMathOperator{\odd}{odd}
\DeclareMathOperator{\even}{even}
\DeclareMathOperator{\im}{Im}
\DeclareMathOperator{\id}{id}
\DeclareMathOperator{\ch}{ch}

\DeclareMathOperator{\BU}{BU}

\DeclareMathOperator{\ho}{Hom}

\DeclareMathOperator{\CS}{CS}

\begin{document}
\setlength{\baselineskip}{1.4\baselineskip}
\theoremstyle{definition}
\newtheorem{defi}{Definition}
\newtheorem{remark}{Remark}
\newtheorem{coro}{Corollary}
\newtheorem{exam}{Example}
\newtheorem{thm}{Theorem}
\newtheorem{prop}{Proposition}
\newtheorem{lemma}{Lemma}
\numberwithin{equation}{section}
\newcommand{\wt}[1]{{\widetilde{#1}}}
\newcommand{\ov}[1]{{\overline{#1}}}
\newcommand{\wh}[1]{{\widehat{#1}}}
\newcommand{\poin}{Poincar$\acute{\textrm{e }}$}
\newcommand{\deff}[1]{{\bf\emph{#1}}}
\newcommand{\boo}[1]{\boldsymbol{#1}}
\newcommand{\abs}[1]{\lvert#1\rvert}
\newcommand{\norm}[1]{\lVert#1\rVert}
\newcommand{\inner}[1]{\langle#1\rangle}
\newcommand{\poisson}[1]{\{#1\}}
\newcommand{\biginner}[1]{\Big\langle#1\Big\rangle}
\newcommand{\set}[1]{\{#1\}}
\newcommand{\Bigset}[1]{\Big\{#1\Big\}}
\newcommand{\BBigset}[1]{\bigg\{#1\bigg\}}
\newcommand{\dis}[1]{$\displaystyle#1$}
\newcommand{\V}{\mathcal{V}}
\newcommand{\R}{\mathbb{R}}
\newcommand{\N}{\mathbb{N}}
\newcommand{\Z}{\mathbb{Z}}
\newcommand{\Q}{\mathbb{Q}}
\newcommand{\h}{\mathbb{H}}
\newcommand{\g}{\mathfrak{g}}
\newcommand{\C}{\mathbb{C}}
\newcommand{\RRR}{\mathscr{R}}
\newcommand{\DDD}{\mathscr{D}}
\newcommand{\so}{\mathfrak{so}}
\newcommand{\gl}{\mathfrak{gl}}
\newcommand{\LL}{\mathcal{L}}
\newcommand{\BB}{\mathcal{B}}
\newcommand{\CC}{\mathcal{C}}
\newcommand{\HH}{\mathcal{H}}
\newcommand{\G}{\mathcal{G}}
\newcommand{\sss}{\mathbb{S}}
\newcommand{\E}{\mathcal{E}}
\newcommand{\EE}{\mathscr{E}}
\newcommand{\UU}{\mathcal{U}}
\newcommand{\F}{\mathcal{F}}
\newcommand{\cdd}[1]{\[\begin{CD}#1\end{CD}\]}
\normalsize
\title{On differential characteristic classes}
\author{Man-Ho Ho}
\address{Department of Mathematics\\ Hong Kong Baptist University}
\email{homanho@hkbu.edu.hk}
\subjclass[2010]{Primary 57R20, 19L50, 53C08}
\maketitle
\nocite{*}
\begin{abstract}
In this paper we give explicit formulas of differential characteristic classes of principal
$G$-bundles with connections and prove their expected properties. In particular, we obtain
explicit formulas for differential Chern classes, differential Pontryagin classes and
differential Euler class. Furthermore, we show that the differential Chern class is the
unique natural transformation from (Simons-Sullivan) differential $K$-theory to
(Cheeger-Simons) differential characters that is compatible with curvature and characteristic
class. We also give the explicit formula for the differential Chern class on Freed-Lott
differential $K$-theory. Finally we discuss the odd differential Chern classes.
\end{abstract}
\tableofcontents
\section{Introduction}

Differential characteristic classes of principal $G$-bundles with connections are secondary
characteristic classes which refine primary characteristic classes. A famous example is given
in \cite{CS74}, where the transgression form lives in the total space. In \cite{CS85}
differential characters are defined with the motivation of defining secondary characteristic
classes living in the base space. The differential characteristic classes constructed in
\cite{CS85} uses universal bundles and universal connections. In particular, we have
differential Chern classes, differential Pontryagin classes and differential Euler class (see
also \cite{DHZ00} for using simplicial method to construct differential Chern classes). Since
a refinement of topological $K$-theory was not available at that time, differential Chern
classes were not considered as natural transformation between refinements of the
corresponding cohomology theories.

In recent years differential $K$-theory--the differential extension of topological
$K$-theory--has received extensive study by its motivation in geometry, topology and
theoretical physics. In \cite{B10} the differential Chern classes are defined on a model of
differential $K$-theory which consists of vector bundles with connections and odd forms,
where the form part has a slightly different additive structure when compared to \cite{FL10}.
The total differential Chern class defined in \cite{B10} respects direct sum, but the ring
structure of this $K$-theory is lost. In \cite{B09} the differential Chern classes is defined
as a natural transformation between the differential extensions of topological $K$-theory and
ordinary cohomology (regarded as functors) which are only assumed to satisfy the axioms of
differential cohomology given in \cite{BS10}. One of the advantages of this approach is the
independence of the construction of the differential Chern classes on a particular model of
differential extension of $K$-theory \cite{HS05, BS09, SS10, FL10} and of ordinary cohomology
\cite{B08, CS85}. Roughly speaking, the differential Chern classes in \cite{B09} are defined
by approximating the classifying space of $K^0$ by a sequence of compact manifolds satisfying
some nice properties and pulling back some universal classes. When working with a particular
model of differential $K$-theory, it would be nice to have explicit formulas for the
differential Chern classes.

In \cite{BB14} differential characters are extended to smooth spaces, and the group of
smooth singular cycles is replaced by the group of diffeomorphism classes of smooth maps
from stratifolds to smooth spaces, which is much more geometric in nature. Inspired by
\cite{BB14} we give explicit formulas for differential characteristic classes for principal
$G$-bundle with connections, where $G$ is a Lie group with finitely many components. The
construction and the proofs of the expected properties do not appeal to universal bundles
and universal connections. Moreover, we give an ``absolute" interpretation of the necessarily
closed $(2k-1)$-form $\alpha$ in the formula, which is a pullback of the transgression form
constructed in \cite{CS74}. Such an interpretation of $\alpha$ is not available for
differential characters in general.

Since all the existing models of differential $K$-theory are isomorphic by unique
isomorphisms \cite[Theorem 3.10]{BS10} and the explicit isomorphisms between different
models of even differential $K$-theory are known \cite{K08a, H12, H13a}, it suffices to
define differential Chern classes in any one of these models. We prove that the explicit
formula of differential Chern classes induces a natural transformation from Simons-Sullivan
differential $K$-theory to differential characters. We also give an explicit formula of
differential Chern classes defined on Freed-Lott differential $K$-theory, where we do not
make use of the explicit isomorphisms. Finally we discuss the odd differential Chern classes
on the odd counterpart of Simons-Sullivan differential $K$-theory.

The paper is organized as follows. Section 2 contains the all the necessary background
materials, and the main results are proved in Section 3.

\section*{Acknowledgement}

The author would like to thank Thomas Schick for the valuable suggestions and comments at
the beginning of this work, and the referee for helpful comments.

\section{Background materials}\setcounter{equation}{0}

Throughout this paper, $A$ is a proper subring of $\R$, $G$ is a Lie group with finitely many
components, and $X$ is a smooth space \cite[Definition 2.2]{BB14}.

\subsection{Geometric chains}

In this subsection we recall some notions in \cite{BB14}. For $n\in\N_0$, let $(C(X),
\partial)$ be the complex of smooth singular $n$-chains on $X$ with integral coefficients.
Denote by $Z_n(X)$ and $B_n(X)$ the subgroups of $n$-cycles and $n$-boundaries respectively.
The space of smooth $n$-forms on $X$ is denoted by $\Omega^n(X)$. A smooth singular chain
$y\in C_n(X)$ is said to be thin \cite[Definition 3.1]{BB14} if for all $\omega\in\Omega^n(X)$,
we have
$$\int_y\omega=0.$$
Denote by $S_n(X)\subseteq C_n(X)$ the subgroup of thin $n$-chains on $X$. Denote by
$[c]_{S_n}$ the equivalence class of $c\in C_n(X)$ in \dis{\frac{C_n(X)}{S_n(X)}}. Note that
$\partial S_{n+1}(X)\subseteq S_n(X)$, so we have a homomorphism
$$\frac{Z_n(X)}{\partial S_{n+1}(X)}\to\frac{C_n(X)}{S_n(X)}.$$
Denote by $[z]_{\partial S_{n+1}}$ the equivalence class of $z\in Z_n(X)$ in \dis{\frac{Z_n(X)}
{\partial S_{n+1}(X)}}.

We recall the definition and basic properties of geometric chains \cite[Chapter 4]{BB14}.
Let
$$\CC_n(X)=\set{[f:M^n\to X]|f:M^n\to X\textrm{ is a smooth map}}$$
be the abelian semigroup of diffeomorphism classes of smooth maps $f:M^n\to X$, where $M^n$
is an oriented compact $n$-dimensional regular $p$-stratifold with boundary \cite{K10}.
Elements in $\CC_n(X)$ are called geometric chains. The boundary operator $\partial:\CC_n(X)
\to\CC_{n-1}(X)$ is given by restriction to the geometric boundary. Define
\begin{displaymath}
\begin{split}
\LL_n(X)&=\set{\xi\in\CC_n(X)|\partial\xi=0},\\
\BB_n(X)&=\set{\xi\in\CC_n(X)|\exists\beta\in\CC_{n+1}(X)\mbox{ such that }\partial\beta
=\xi}.
\end{split}
\end{displaymath}
Elements in $\LL_n(X)$ are called geometric cycles and elements in $\BB_n(X)$ are
called geometric boundaries. Denote by \dis{\HH_k(X):=\frac{\LL_n(X)}{\BB_n(X)}} the
corresponding cohomology group. Define a homomorphism \dis{\psi_n:\LL_n(X)\to\frac{Z_n(X)}
{\partial S_{n+1}(X)}} by
\begin{equation}\label{eq 2.1}
\psi_n([f:M\to X])=[f_*(c)]_{\partial S_{n+1}},
\end{equation}
where $c\in Z_n(M)$ is an $n$-cycle representing the fundamental class of $M$. $\psi_n$
is independent of the choices of $c$ by \cite[Remark 3.2]{BB14}.

Note that $\HH_k(X)\cong H_k(X)$ via the map $\psi_n$ \cite[Theorem 20.1]{K10}. Henceforth
we write $[\zeta]_{\partial S_{n+1}}$ for $\psi_n(\zeta)$ for $\zeta=[f:M\to X]\in\LL_n(X)$.
Similar convention applies to elements in $\CC_n(X)$.
\begin{lemma}\label{lemma 1}\cite[Lemma 4.2]{BB14}
There exist homomorphisms $\zeta:C_{n+1}(X)\to\CC_{n+1}(X)$, $a:C_n(X)\to C_{n+1}(X)$
and $y:C_{n+1}(X)\to Z_{n+1}(X)$ such that
\begin{eqnarray}
(\partial\zeta)(c)&=&\zeta(\partial c),\nonumber\\
\big[\zeta(c)\big]_{S_{n+1}}&=&\big[c-a(\partial c)-\partial a(c+y(c))\big]_{S_{n+1}},\nonumber\\
\big[\zeta(z)\big]_{\partial S_{n+1}}&=&\big[z-\partial a(z)\big]_{\partial S_{n+1}}.\label{eq 2.2}
\end{eqnarray}
for all $c\in C_{n+1}(X)$ and all $z\in Z_{n+1}(X)$.
\end{lemma}

\subsection{Cheeger-Simons differential characters}

In this subsection we recall Cheeger-Simons differential characters \cite{CS85} (see also
\cite[Chapter 5]{BB14}).

Let $k\geq 1$. A degree $k$ differential character $f$ with coefficients in $A$ is a group
homomorphism $f:Z_{k-1}(X)\to\R/A$ such that there exists a fixed $\omega_f\in\Omega^k(X)$
such that for all $c_k\in C_k(X)$,
$$f(\partial c)=\int_c\omega_f\mod A.$$
The abelian group of degree $k$ differential characters is denoted by $\wh{H}^k(X; \R/A)$.
It is easy to see that $\omega_f$ is a closed $k$-form with periods in $A$ and is uniquely
determined by $f\in\wh{H}^k(X; \R/A)$.

In the following diagram
\begin{equation}\label{eq 2.3}
\xymatrix{\scriptstyle 0 \ar[dr] & \scriptstyle & \scriptstyle & \scriptstyle & \scriptstyle 0
\\ & \scriptstyle H^{k-1}(X; \R/A) \ar[rr]^{-B} \ar[dr]^{i_1} & \scriptstyle & \scriptstyle
H^k(X; A) \ar[ur] \ar[dr]^r & \scriptstyle \\ \scriptstyle H^{k-1}(X; \R) \ar[ur]^{\alpha}
\ar[dr]_{\beta} & \scriptstyle & \scriptstyle\wh{H}^k(X; \R/A) \ar[ur]^{\delta_2}
\ar[dr]^{\delta_1} & \scriptstyle & \scriptstyle H^k(X; \R) \\ \scriptstyle & \scriptstyle
\frac{\Omega^{k-1}(X)}{\Omega^{k-1}_{A}(X)} \ar[rr]_{d} \ar[ur]^{i_2} & \scriptstyle &
\scriptstyle\Omega^k_{A}(X) \ar[ur]^s \ar[dr] & \scriptstyle \\ \scriptstyle 0 \ar[ur] &
\scriptstyle & \scriptstyle & \scriptstyle & \scriptstyle 0}
\end{equation}
the diagonal sequences are exact, and every triangle and square commutes
\cite[Theorem 1.1]{CS85}: Here $\Omega_A^k(X)$ denotes the group of closed $k$-forms on
$X$ with periods in $A$. The maps are defined as follows: $r$ is induced by
$A\hookrightarrow\R$,
$$i_1([z])=z|_{Z_{k-1}(X)},~~~i_2(\omega)=\ov{\omega}|_{Z_{k-1}(X)},~~~\delta_1(f)
=\omega_f\textrm{ and }\delta_2(f)=[c],$$
where $[c]\in H^k(X; A)$ is the unique cohomology class satisfying $r[c]=[\omega_f]$. In
literatures $\delta_1(f)$ is called the curvature of $f$, and $\delta_2(f)$ is called
the characteristic class of $f$.

The basic setup of differential characteristic classes is the following. Let $I^k(G)$ be
the ring of invariant polynomials of degree $k$ on $G$, and $w:I^k(G)\to H^{2k}(BG; \R)$
the
Weil homomorphism. Define
$$K^{2k}(G, A)=\set{(P, u)\in I^k(G)\times H^{2k}(BG; A)|w(P)=r(u)}.$$
Differential characteristic classes can be regarded as the unique natural transformation
$S:K^{2k}(G; A)\to\wh{H}^{2k}(X; \R/A)$ which makes the following diagram commutes
$$\xymatrix{ & \wh{H}^{2k}(X; \R/A) \ar[dr]^{(\delta_1, \delta_2)} & \\ K^{2k}(G; \Z)
\ar[rr]_{w\times c_A} \ar[ur]^S & & R^{2k}(X; A)}$$
where $c_A:H^{2k}(BG; A)\to H^{2k}(X; A)$ is induced by a classifying map $X\to BG$ for
a principal $G$-bundle $P\to X$ with connection $\theta$, and
$$R^k(X; A):=\set{(\omega, u)\in\Omega^k_A(X)\times H^k(X; A)|r(u)=[w]}.$$

Let $f\in\wh{H}^k(X; \R/A)$. If $c\in S_k(X)$ is a thin chain, then
$$f(\partial c)=\int_c\omega=0\mod A.$$
This property of differential characters is refereed as thin invariance
\cite[Remark 5.2]{BB14}. Note that the results in \cite{BB14} hold if $\Z$ is replaced
by $A$.

\section{Main Results}\setcounter{equation}{0}

\subsection{Differential characteristic classes}\label{3.1}

First of all we prove the uniqueness of differential characteristic classes.

\begin{prop}\label{prop 1}
Let $(P, u)\in K^{2k}(G, A)$, where $k\geq 1$. For each principal $G$-bundle $\pi:E\to X$
with a connection $\theta$, if there exists $S_{P, u}(E, \theta)\in\wh{H}^{2k}(X; \R/A)$
such that $S_{P, u}(E, \theta)$ is natural and $\delta_1(S_{P, u}(E, \theta))=P(\Omega)$,
where $\Omega$ is the curvature of $\theta$, then $S_{P, u}(E, \theta)$ is unique.
\end{prop}
\begin{proof}
Let $z\in Z_{2k-1}(X)$. By (\ref{eq 2.2}) we have $[\zeta(z)]_{\partial S_{2k}}=[z-\partial
a(z)]_{\partial S_{2k}}$. Then
\begin{equation}\label{eq 3.1}
S_{P, u}(E, \theta)(z)=S_{P, u}(E, \theta)([\zeta(z)]_{\partial S_{2k}})+S_{P, u}(E, \theta)
(\partial a(z)).
\end{equation}
By the definition of differential character and the assumption $\delta_1(S_{P, u}(E,
\theta))=P(\Omega)$, we have
\begin{displaymath}
\begin{split}
S_{P, u}(E, \theta)(\partial a(z))&=\int_{a(z)}\delta_1(S_{P, u}(E, \theta))\mod A\\
&=\int_{a(z)}P(\Omega)\mod A.
\end{split}
\end{displaymath}
Write $\zeta(z)=[g:M\to X]\in\LL_{2k-1}(X)$. Since $S_{P, u}(E, \theta)$ is assumed to be
natural, it follows from (\ref{eq 2.1}) that
\begin{displaymath}
\begin{split}
S_{P, u}(E, \theta)([\zeta(z)]_{\partial S_{2k}})&=S_{P, u}(E, \theta)([g:M\to X])\\
&=S_{P, u}(E, \theta)(g_*(c))\\
&=g^*S_{P, u}(E, \theta)(c)\\
&=S_{P, u}(g^*E, g^*\theta)(c),
\end{split}
\end{displaymath}
where $c\in Z_{2k-1}(M)$ is a cycle representing the fundamental class of $M$. It follows
from these two observations that (\ref{eq 3.1}) becomes
\begin{equation}\label{eq 3.2}
S_{P, u}(E, \theta)(z)=S_{P, u}(g^*E, g^*\theta)(c)+\int_{a(z)}P(\Omega)\mod A.
\end{equation}
Since $\dim(M)=2k-1$, it follows that $\delta_2(S_{P, u}(g^*E, g^*\theta))=0$. Thus there
exists \dis{\alpha\in\frac{\Omega^{2k-1}(M)}{\Omega^{2k-1}_\Z(M)}} such that
$S_{P, u}(g^*E, g^*\theta)=i_2(\alpha)$. Thus (\ref{eq 3.2}) becomes
\begin{equation}\label{eq 3.3}
S_{P, u}(E, \theta)(z)=\int_M\alpha+\int_{a(z)}P(\Omega)\mod A.
\end{equation}
Thus $S_{P, u}(E, \theta)(z)$ is uniquely determined by (\ref{eq 3.3}).
\end{proof}
We take (\ref{eq 3.3}) as the definition of $S_{P, u}(E, \theta)$.

The following proposition shows that (\ref{eq 3.3}) is independent of the choices made in
the proof of Proposition \ref{prop 1}, and it defines a differential character.

\begin{prop}\label{prop 2}
Let $(P, u)\in K^{2k}(G, A)$, where $k\geq 1$. For each principal $G$-bundle $\pi:E\to X$
with a connection $\theta$. the differential characteristic class $S_{P, u}(E, \theta)$
defined in (\ref{eq 3.3}) is independent of the choices made in Proposition \ref{prop 1},
i.e., for $z\in Z_{2k-1}(X)$, if $\zeta'(z)=[g':M\to X]\in\LL_{2k-1}(X)$ and $a'(z)\in
C_{2k}(X)$ are such that $[\zeta'(z)]_{\partial S_{2k}}=[z-\partial a'(z)]_{\partial S_{2k}}$,
and $\alpha'\in\Omega^{2k-1}(M')$ is such that $(g')^*S_{P, u}(E, \theta)=i_2(\alpha')$,
then
$$S_{P, u}(E, \theta)(z)=\int_{M'}\alpha'+\int_{a'(z)}P(\Omega)\mod A.$$
Moreover, $S_{P, u}(E, \theta)$ given by (\ref{eq 3.3}) defines a differential character
in $\wh{H}^{2k}(X; \R/A)$.
\end{prop}
The proof is virtually the same as \cite[Lemma 5.13]{BB14}.
\begin{proof}
Note that $[z]=[\zeta(z)]_{\partial S_{2k}}=[\zeta'(z)]_{\partial S_{2k}}$, and
therefore $\zeta'(z)-\zeta(z)=\partial\beta(z)$ for some $\beta(z)\in\CC_{2k}(X)$.
Thus
\begin{displaymath}
\begin{split}
[\partial a(z)-\partial a'(z)]_{\partial S_{2k}}&=[\zeta(z)-\zeta'(z)]_{\partial S_{2k}}\\
&=[\partial\beta(z)]_{\partial S_{2k}}\\
&=\partial[\beta(z)]_{S_{2k}}\\
\Rightarrow0&=\partial[a(z)-a'(z)-\beta(z)]_{\partial S_{2k}}.
\end{split}
\end{displaymath}
Thus there exists $w(z)\in Z_{2k}(X)$ such that
\begin{equation}\label{eq 3.4}
[a(z)-a'(z)-w(z)]_{S_{2k}}=[\beta(z)]_{S_{2k}}.
\end{equation}
Write $\beta(z)=[G:N\to X]$, where, by definition, $N$ is a $2k$-dimensional compact
oriented $p$-stratifold with boundary $\partial N=M'\sqcup\bar{M}$ with $g=G|_M$ and
$g'=G|_{M'}$. Since $H^{2k}(N; A)=0$ and $G^*S_{P, u}(E, \theta)\in\wh{H}^{2k}(N; \R/A)$,
it follows that $\delta_2(G^*S_{P, u}(E, \theta))=0$. Thus there exists
\dis{\chi\in\frac{\Omega^{2k-1}(N)}{\Omega^{2k-1}_A(N)}} such that $G^*S_{P, u}(E, \theta)
=i_2(\chi)$. Note that
\begin{eqnarray}
i_2(\alpha')-i_2(\alpha)&=&(g')^*S_{P, u}(E, \theta)-g^*S_{P, u}(E, \theta)\nonumber\\
&=&(G|_{\partial N})^*S_{P, u}(E, \theta)\nonumber\\
&=&(G^*S_{P, u}(E, \theta))|_{\partial N}\nonumber\\
&=&i_2(\chi)|_{\partial N}\nonumber\\
\Rightarrow\alpha'-\alpha&=&\chi|_{\partial N}+\eta\label{eq 3.5}
\end{eqnarray}
for some $\eta\in\Omega^{2k-1}_A(\partial N)$. Thus
\begin{displaymath}
\begin{split}
&\quad\bigg(\int_{M'}\alpha'+\int_{a'(z)}P(\Omega)\bigg)-S_{P, u}(E, \theta)(z)\mod A\\
&=\int_{M'}\alpha'+\int_{a'(z)}P(\Omega)-\int_M\alpha-\int_{a(z)}P(\Omega)\mod A\\
&=\int_{\partial N}(\alpha'-\alpha)+\int_{a'(z)-a(z)}P(\Omega)\mod A\\
&=\int_{\partial N}(\chi+\eta)+\int_{-w(z)}P(\Omega)+\int_{-[\beta(z)]_{S_{2k}}}
P(\Omega)\mod A\\
\end{split}
\end{displaymath}
by (\ref{eq 3.4}) and (\ref{eq 3.5}). Since $\eta\in\Omega^{2k-1}_A(\partial N)$ and
$P(\Omega)\in\Omega^{2k}_A(X)$, we have
\begin{displaymath}
\begin{split}
&\quad\bigg(\int_{M'}\alpha'+\int_{a'(z)}P(\Omega)\bigg)-S_{P, u}(E, \theta)(z)\\
&=\int_{\partial N}\chi+\int_{-[\beta(z)]_{S_{2k}}}P(\Omega)\mod A\\
&=\int_Nd\chi+\int_{-[\beta(z)]_{S_{2k}}}P(\Omega)\mod A\\
&=\int_NG^*P(\Omega)+\int_{-[\beta(z)]_{S_{2k}}}P(\Omega)\mod A\\
&=\int_{G_*[N]_{S_{2k}}-[\beta(z)]_{S_{2k}}}P(\Omega)\mod A\\
&=0
\end{split}
\end{displaymath}
since $[\beta(z)]=G_*[N]_{S_{2k}}$, and the third equality follows from the commutativity
of the lower triangle of (\ref{eq 2.3}).

Since $\zeta$ and $a$ in (\ref{eq 3.3}) are homomorphisms by Lemma \ref{lemma 1}, it
follows that $S_{P, u}(E, \theta):Z_{2k-1}(X)\to\R/A$ is a homomorphism.

To prove $S_{P, u}(E, \theta)$ is a differential character, we need to show that for
$z=\partial c$, where $c\in C_{2k}(X)$, we have
\begin{equation}\label{eq 3.6}
S_{P, u}(E, \theta)(\partial c)=\int_cP(\Omega)\mod A.
\end{equation}
The proof of (\ref{eq 3.6}) is essentially the same as (a) of the proof of
\cite[Theorem 5.14]{BB14}. Thus $S_{P, u}(E, \theta)\in\wh{H}^{2k}(X; \R/A)$.
\end{proof}

The following proposition shows the expected properties of differential characteristic
classes.

\begin{thm}\label{thm 1}
Let $(P, u)\in K^{2k}(G, A)$, where $k\geq 1$. For each principal $G$-bundle $\pi:E\to X$
with a connection $\theta$, we have
\begin{enumerate}
  \item $\delta_1(S_{P, u}(E, \theta))=P(\Omega)$,
  \item $\delta_2(S_{P, u}(E, \theta))=u(E)$, and
  \item if $f:Y\to X$ is any smooth map, then
        $$S_{P, u}(f^*E, f^*\theta)=f^*S_{P, u}(E, \theta).$$
\end{enumerate}
\end{thm}
\begin{proof}~~~~~~~~~~~~~~~~~~~~~~~~
\begin{enumerate}
  \item It follows directly from (\ref{eq 3.6}).
  \item Recall that a cocycle $u_k\in Z^{2k}(X; A)$ representing $\delta_2(S_{P, u}(E,
        \theta))\in H^{2k}(X; A)$ is defined by
        $$u_k(c)=\int_cP(\Omega)-T(\partial c),$$
        where $T\in\ho(Z_{2k-1}(X), \R)$ is a lift of $S_{P, u}(E, \theta)$, i.e.,
        $S_{P, u}(E, \theta)(z)=T(z)\mod A$ for all $z\in Z_{2k-1}(X)$. \emph{A priori}
        $u_k$ is a real cocycle. Since
        \begin{displaymath}
        \begin{split}
        u_k(c)&=\int_cP(\Omega)-T(\partial c)\mod A\\
        &=S_{P, u}(E, \theta)(\partial c)-S_{P, u}(E, \theta)(\partial c)\mod A\\
        &=0\mod A,
        \end{split}
        \end{displaymath}
        $u_k$ is indeed an $A$-cocycle. Note that $u_k$ depends on the lift $T$, but
        its cohomology class does not. Since \dis{u_k(z)=\int_zP(\Omega)} for all $z\in
        Z_{2k}(X)$, it follows from the uniqueness of the de Rham theorem that $u_k$
        represents $u(E)$.
  \item The proof is similar to (c) of \cite[Theorem 5.14]{BB14}. Let $z\in Z_{2k-1}(X)$.
        By (\ref{eq 2.2}) we have $[\zeta(z)]_{\partial S_{2k}}=[z-\partial a(z)]_{\partial
        S_{2k}}$, where $\zeta(z)\in\LL_{2k-1}(X)$ and $a(z)\in C_{2k}(X)$. Write $\zeta(z)
        =[g:M\to X]$. By (\ref{eq 3.3}), we have
        \begin{equation}\label{eq 3.7}
        S_{P, u}(f^*E, f^*\theta)(z)=\int_M\alpha+\int_{a(z)}P(f^*\Omega)\mod A,
        \end{equation}
        where \dis{\alpha\in\frac{\Omega^{2k-1}(M)}{\Omega^{2k-1}_A(M)}} is the unique closed
        form such that
        \begin{equation}\label{eq 3.8}
        i_2(\alpha)=S_{P, u}(g^*f^*E, g^*f^*\theta)=S_{P, u}((f\circ g)^*E, (f\circ g)^*\theta).
        \end{equation}
        Since $f^*S_{P, u}(E, \theta)(z)=S_{P, u}(E, \theta)(f_*z)$, we compute $f_*z$.
        Define by $\zeta(f_*z):=f_*\zeta(z)=[f\circ g:M\to Y]$ and $a(f_*z):=f_*a(z)$. Then
        \begin{displaymath}
        \begin{split}
        [f_*z-\partial a(f_*z)]_{\partial S_{2k}}&=f_*[z-\partial a(z)]_{\partial S_{2k}}\\
        &=f_*[\zeta(z)]_{\partial S_{2k}}\\
        &=[f_*\zeta(z)]_{\partial S_{2k}}.
        \end{split}
        \end{displaymath}
        It follows that
        \begin{displaymath}
        \begin{split}
        &\quad f^*S_{P, u}(E, \theta)(z)\\
        &=S_{P, u}(E, \theta)(f_*z)\\
        &=S_{P, u}(E, \theta)([f_*\zeta(z)]_{\partial S_{2k}})+S_{P, u}(E, \theta)(\partial
        a(f_*(z))\\
        &=S_{P, u}(E, \theta)(f_*[\zeta(z)]_{\partial S_{2k}})+\int_{f_*a(z)}P(\Omega)\mod A\\
        &=S_{P, u}(E, \theta)(f_*(g_*(c)))+\int_{a(z)}f^*P(\Omega)\mod A.\\
        \end{split}
        \end{displaymath}
        Thus
        \begin{equation}\label{eq 3.9}
        f^*S_{P, u}(E, \theta)=(f\circ g)^*S_{P, u}(E, \theta)(c)+\int_{a(z)}P(f^*\Omega)
        \mod A.
        \end{equation}
        Since $(f\circ g)^*S_{P, u}(E, \theta)\in\wh{H}^{2k}(M; \R/A)$ and $\dim(M)=2k-1$,
        it follows one of the exact sequences of (\ref{eq 2.3}) that there exists a unique
        closed form \dis{\beta\in\frac{\Omega^{2k-1}(M)}{\Omega^{2k-1}_A(M)}} such that
        $$i_2(\beta)=S_{P, u}((f\circ g)^*E, (f\circ g)^*\theta).$$
        It follows from (\ref{eq 3.8}) that $\beta-\alpha\in\Omega^{2k-1}_A(M)$. Thus
        (\ref{eq 3.9}) becomes
        \begin{displaymath}
        \begin{split}
        f^*S_{P, u}(E, \theta)(z)&=\int_M\alpha+\int_{a(z)}P(f^*\Omega)\mod A\\
        &=S_{P, u}(f^*E, f^*\theta)(z).\qedhere
        \end{split}
        \end{displaymath}
\end{enumerate}
\end{proof}

\begin{remark}\label{remark 1}
Since the $S_{P, u}(E, \theta)$ given by (\ref{eq 3.3}) satisfies Theorem \ref{thm 1}, it
follows from Proposition \ref{prop 1} that $S_{P, u}(E, \theta)$ is indeed the unique such
differential character. In particular, this gives a proof of \cite[Theorem 2.2]{CS85} without
using universal bundles and universal connections.
\end{remark}

The following lemma gives an ''absolute" interpretation of the form $\alpha$ in (\ref{eq 3.3}).

\begin{coro}\label{coro 1}
Let $(P, u)\in K^{2k}(G, A)$, where $k\geq 1$. For each principal $G$-bundle $\pi:E\to X$
with a connection $\theta$, if
$$(\pi^*S_{P, u}(E, \theta))(z)=\int_M\alpha+\int_{a(z)}P(\pi^*\Omega)\mod\Z,$$
for $z\in Z_{2k-1}(E)$, then
$$\alpha=g^*\TP(\theta)\in\frac{\Omega^{2k-1}(M)}{\Omega^{2k-1}_\Z(M)},$$
where $[g:M\to E]=\zeta(z)$, and $\TP(\theta)\in\Omega^{2k-1}(E)$ is the transgression
form defined in \cite[\S3]{CS74}.
\end{coro}
\begin{proof}~~~~~~~~~~~~~~~~~~~~
By the naturality of $S_{P, u}$, we have
\begin{eqnarray}
\pi^*S_{P, u}(E, \theta)(z)&=&S_{P, u}(\pi^*E, \pi^*\theta)(z)\nonumber\\
&=&\int_M\alpha+\int_{a(z)}P(\pi^*\Omega)\mod\Z,\label{eq 3.10}
\end{eqnarray}
where $\alpha\in\Omega^{2k-1}(M)$ is, unique up to a closed $(2k-1)$-form with periods in
$\Z$, such that $i_2(\alpha)=g^*S_{P, u}(\pi^*E, \pi^*\theta)$. Note that
\begin{eqnarray}
\int_z\TP(\theta)&=&\int_{g_*c}\TP(\theta)+\int_{a(z)}d\TP(\theta)\nonumber\\
&=&\int_Mg^*\TP(\theta)+\int_{a(z)}\pi^*P(\Omega)\nonumber\\
&=&\int_Mg^*\TP(\theta)+\int_{a(z)}P(\pi^*\Omega),\label{eq 3.11}
\end{eqnarray}
where $c\in Z_{2k-1}(M)$ represents the fundamental class of $M$. Here the second equality
follows from \cite[Proposition 3.2]{CS74} and the third equality follows from the fact that
$P(\Omega)\in\Omega(\pi^*E)$ is horizontal. Since
$$\pi^*(S_{P, u}(E, \theta))=i_2(\TP(\theta))$$
by \cite[Proposition 2.8]{CS85}, by (\ref{eq 3.10}) and (\ref{eq 3.11}) we have
\begin{displaymath}
\alpha=g^*\TP(\theta)\in\frac{\Omega^{2k-1}(M)}{\Omega^{2k-1}_\Z(M)}.\qedhere
\end{displaymath}
\end{proof}

One can apply a similar procedure as in Proposition \ref{prop 1} to construct differential
Chern classes, differential Pontryagin classes and differential Euler class. For a
Hermitian bundle $E\to X$ with a unitary connection $\nabla^E$, the $k$-th differential
Chern class $\wh{c}_k(E, \nabla)\in\wh{H}^{2k}(X; \R/\Z)$, where $k\geq 1$, is given by
\begin{equation}\label{eq 3.12}
\wh{c}_k(E, \nabla)(z)=\int_M\alpha+\int_{a(z)}c_k(\nabla)\mod\Z.
\end{equation}
The total differential Chern class $\wh{c}(E, \nabla)$ is defined to be
\begin{equation}\label{eq 3.13}
\wh{c}(E, \nabla):=1+\wh{c}_1(E, \nabla)+\cdots.
\end{equation}
For a Euclidean vector bundle $E\to X$ with a metric connection $\nabla^E$, the $k$-th
differential Pontryagin class $\wh{p}_k(E, \nabla)\in\wh{H}^{4k}(X; \R/\Z)$ is given by
\begin{equation}\label{eq 3.14}
\wh{p}_k(E, \nabla)(z)=\int_M\alpha+\int_{a(z)}p_k(\nabla)\mod\Z,
\end{equation}
and the differential Euler class $\wh{\chi}(E, \nabla)\in\wh{H}^{2n}(X; \R/\Z)$ with
$n=\rk(E)$ is given by
\begin{equation}\label{eq 3.15}
\wh{\chi}(E, \nabla)(z)=\int_M\alpha+\int_{a(z)}\chi(\nabla)\mod\Z.
\end{equation}
In particular, analogous statements of Theorem \ref{thm 1} hold for these differential
characteristic classes, and therefore its uniqueness. One can compare (\ref{eq 3.12}),
(\ref{eq 3.14}) and (\ref{eq 3.15}) with \cite[(4.4)]{CS85}, \cite[(3.3)]{CS85} and
\cite[(5.1)]{CS85}.

\begin{exam}\label{exam 1}
Let $\varepsilon\to X$ be a trivial complex vector bundle with a metric and a unitary flat
connection $d$. For $k\geq 1$, since $\wh{c}_k(\varepsilon, d)$ is natural and its curvature
and characteristic class are zero respectively, it follows from the uniqueness (see
Remark \ref{remark 1}) that $\wh{c}_k(\varepsilon, d)=0$.
\end{exam}

\subsection{Differential Chern classes on $\wh{K}_{\SSSS}$}

In this subsection we show that the differential Chern classes given by (\ref{eq 3.12}) is
the unique natural transformation from Simons-Sullivan differential $K$-theory to
differential characters. We refer the details of $\wh{K}_{\SSSS}$ to \cite{SS10}.

First of all we give a ''relative" interpretation of the form $\alpha$ in
(\ref{eq 3.12}).

\begin{lemma}\label{lemma 2}
Let $E\to X$ be a Hermitian bundle with $X$ compact. If $\nabla^0$ and $\nabla^1$ are two
unitary connections on $E\to X$ and for $i=0, 1$,
$$\wh{c}_k(E, \nabla^i)(z)=\int_M\alpha_i+\int_{a(z)}c_k(\nabla^i)\mod\Z$$
as given in (\ref{eq 3.12}), then
\begin{equation}\label{eq 3.16}
\alpha_1-\alpha_0=g^*Tc_k(\nabla^1, \nabla^0)\in\frac{\Omega^{2k-1}(M)}
{\Omega^{2k-1}_\Z(M)}
\end{equation}
where $Tc_k(\nabla^1, \nabla^0)$ is the transgression form between the $k$-th Chern forms
of $\nabla^1$ and $\nabla^0$, and $\zeta(z)=[g:M\to X]$.
\end{lemma}
\begin{proof}
For $z\in Z_{2k-1}(X)$, we have $[\zeta(z)]_{\partial S_{2k}}=[z-\partial a(z)]_{\partial
S_{2k}}$ by (\ref{eq 2.2}). Note that
\begin{eqnarray}
&\quad& i_2(Tc_k(\nabla^1, \nabla^0))(z)\nonumber\\
&=&\int_zTc_k(\nabla^1, \nabla^0)\mod\Z\nonumber\\
&=&\int_{[\zeta(z)]_{\partial S_{2k}}}Tc_k(\nabla^1, \nabla^0)+\int_{a(z)}dTc_k
(\nabla^1, \nabla^0)\mod\Z\nonumber\\
&=&\int_Mg^*Tc_k(\nabla^1, \nabla^0)+\int_{a(z)}(c_k(\nabla^1)-c_k(\nabla^0))\mod\Z
\label{eq 3.17}
\end{eqnarray}
and
\begin{equation}\label{eq 3.18}
\begin{split}
&\quad\wh{c}_k(E, \nabla^1)(z)-\wh{c}_k(E, \nabla^0)(z)\\
&=\int_M(\alpha_1-\alpha_0)+\int_{a(z)}(c_k(\nabla^1)-c_k(\nabla^0))\mod\Z.
\end{split}
\end{equation}
By the analogue of \cite[Proposition 2.9]{CS85} for vector bundles, we have
$$\wh{c}_k(E, \nabla^1)-\wh{c}_k(E, \nabla^0)=i_2(Tc_k(\nabla^1, \nabla^0)).$$
Thus (\ref{eq 3.16}) follows from (\ref{eq 3.17}) and (\ref{eq 3.18}).
\end{proof}

The following proposition shows that the differential Chern classes is a well defined map
from Simons-Sullivan differential $K$-theory to differential characters.

\begin{prop}\label{prop 3}
Let $X$ be compact. For each $k\geq 1$, the $k$-th differential Chern class $\wh{c}_k:
\wh{K}_{\SSSS}(X)\to\wh{H}^{2k}(X; \R/\Z)$ defined on a generator $\E$ of $\wh{K}_{\SSSS}
(X)$ by
\begin{equation}\label{eq 3.19}
\wh{c}_k(\E):=\wh{c}_k(E, \nabla),
\end{equation}
is a well defined map.
\end{prop}
\begin{proof}
\emph{A priori} $\wh{c}_k(\E)$ is not well defined on the level of generators as the right
hand side of (\ref{eq 3.19}) depends on the choice of $\nabla\in[\nabla]$. Take another
connection $\nabla'\in[\nabla]$. By the definition of $\wh{K}_{\SSSS}(X)$ we have
\dis{\CS(\nabla', \nabla)=0\in\frac{\Omega^{\odd}(X)}{d\Omega^{\even}(X)}}. Thus
$\ch(\nabla')=\ch(\nabla)$, which implies $c_k(\nabla')=c_k(\nabla)$ for all $k\geq 1$. If
we write
$$\wh{c}_k(E, \nabla')(z)=\int_M\alpha'+\int_{a(z)}c_k(\nabla')\mod\Z$$
for $z\in Z_{2k-1}(X)$, it follows from Lemma \ref{lemma 2} that
$$\alpha'-\alpha=g^*Tc_k(\nabla', \nabla)=0\in\frac{\Omega^{2k-1}(M)}{\Omega^{2k-1}_\Z(M)}.$$
Thus $\wh{c}_k(\E)$ is independent of the choice of representative of the connection.

Now we show that $\wh{c}_k$ is a well defined map. Let $\E=\F\in\wh{K}_{\SSSS}(X)$. We
prove that
\begin{equation}\label{eq 3.20}
\wh{c}_k(\E)=\wh{c}_k(\F)
\end{equation}
for all $k\geq 1$. By the definition of $\wh{K}_{\SSSS}(X)$, there exists a structured bundle
$\G$ such that
$$\E\oplus\G\cong\F\oplus\G.$$
By \cite[Corollary 3]{PT13} there exists a structured inverse $\HH$ to $\G$ (which is
proved without using universal bundles and universal connections), i.e., $\HH\oplus\G=
[n]$ for some $n\in\N$. Thus
\begin{displaymath}
\begin{split}
\E\oplus\G\oplus\HH&\cong\F\oplus\G\oplus\HH\nonumber\\
\Rightarrow\E-[n]&\cong\F-[n]\\
\Rightarrow\wh{c}(\E-[n])&=\wh{c}(\F-[n]).
\end{split}
\end{displaymath}
By Example \ref{exam 1}, we have
$$\wh{c}(\E-[n]):=\frac{\wh{c}(\E)}{\wh{c}([n])}=\wh{c}(\E)$$
and similarly we have $\wh{c}(\F-[n])=\wh{c}(\F)$. Thus (\ref{eq 3.20}) holds.
\end{proof}

Denote by $s_k(x_1, \ldots, x_n)$ the $k$-th elementary symmetric function and $P_k(x_1,
\ldots, x_n)$ the $k$-th Newton's function of $n$ variables $x_1, \ldots, x_n$ , i.e.,
$$s_k(x_1, \ldots, x_n)=\sum_{i_1<\cdots<i_k}x_{i_1}\cdots x_{i_k},\qquad P_k(x_1, \ldots,
x_n)=\sum_{j=1}^nx_j^k.$$
For any $k\leq n$, we have \cite{M92}
$$P_k+\sum_{j=1}^{k-1}P_{k-j}s_j+(-1)^ks_k=0.$$
It follows from the above identity that we can express each $s_k$ in terms of the $P_j$'s
and vice versa.

Let $k\geq 1$. Define a map $s_k:\Omega^{\even}(X)\to\Omega^{2k}(X)$ as follows. Write
\dis{\omega=\sum_{j=0}\frac{1}{j!}\omega_{[2j]}\in\Omega^{\even}(X)}, where \dis{\frac{1}{j!}
\omega_{[2j]}\in\Omega^{2j}(X)} is the degree $2j$ component of $\omega$. Write
\dis{\omega_{[2j]}=\frac{1}{j!}\omega'_{[2j]}}, where $\omega'_{[2j]}:=j!\omega_{[2j]}$.
Define
$$s_k(\omega)=s_k\bigg(\frac{1}{2!}\omega'_{[2]}, \ldots, \frac{1}{j!}\omega'_{[2j]},
\ldots\bigg),$$
to be the $k$-th elementary symmetric function of the $\omega_{[2j]}$'s. Note that for
each $k\geq 1$, $s_k(\omega)$ can be given in terms of the Newton's functions $P_\ell
(\omega'_{[2]}, \ldots, \omega'_{[2j]}, \ldots)$. For example, $s_k(\ch(\nabla))=c_k
(\nabla)$, the $k$-th Chern form of $\nabla$.

Let $\Omega^\bullet_{\BU}(X)=\set{\omega\in\Omega^\bullet_{d=0}|[\omega]\in\im(\ch^\bullet:
K^{-(\bullet\mod 2)}(X)\to H^\bullet(X; \Q))}$, where $\bullet\in\set{\even, \odd}$.
\begin{lemma}\label{lemma 3}
$s_k(\Omega^{\even}_{\BU}(X))\subseteq\Omega^{2k}_\Z(X)$.
\end{lemma}
\begin{proof}
Let $\omega\in\Omega^{\even}_{\BU}(X)$. Then $\omega=\ch(\nabla)+d\alpha$, where $\nabla$
is a unitary connection on a Hermitian bundle over $X$ and $\alpha\in\Omega^{\odd}(X)$.
Write \dis{\ch(\nabla)=\sum_{j=0}^n\ch_j(\nabla)}, where $\ch_j(\nabla)$ is the degree
$2j$ component of $\ch(\nabla)$. Thus for each $j\geq 1$, $\omega_{[2j]}=\ch_j(\nabla)+d
(\alpha_{[2j-1]})$. Note that
\begin{displaymath}
\begin{split}
s_k(\alpha)&=s_k\bigg(\frac{1}{2!}\omega_{[2]}', \ldots, \frac{1}{j!}\omega'_{[2j]}, \ldots
\bigg)\\
&=c_k(\nabla)+Q_k\bigg(\frac{1}{2!}\omega_{[2]}', \ldots, \frac{1}{j!}\omega'_{[2j]}, \ldots
\bigg).
\end{split}
\end{displaymath}
Here \dis{Q_k\bigg(\frac{1}{2!}\omega_{[2]}', \ldots, \frac{1}{j!}\omega'_{[j]}, \ldots
\bigg)} is a sum of rational multiples of $\ch_j(\nabla)^m\wedge(d\alpha_{[2i-1]})^q$ for
some $m\geq 0$ and $q\geq 1$. Since each $\ch_j(\nabla)^m\wedge(d\alpha_{[2i-1]})^q$ is
exact, it has period $0$. Together with the fact that $c_k(\nabla)\in\Omega^{2k}_\Z(X)$,
we have $s_k(\omega)\in\Omega^{2k}_\Z(X)$.
\end{proof}

By Proposition \ref{prop 3}, we can reformulate Theorem \ref{thm 1} as follows.

\begin{coro}\label{coro 2}
For each $k\geq 0$, there exists a unique natural transformation $\wh{c}_k:\wh{K}_{\SSSS}
(\ast)\to\wh{H}^{2k}(\ast; \R/\Z)$ such that it is compatible with curvature and
characteristic class, i.e., for each compact $X$, the following diagrams commute.
\cdd{\wh{K}_{\SSSS}(X) @>\wh{c}_k>> \wh{H}^{2k}(X; \R/\Z) \\ @V\ch_{\wh{K}_{\SSSS}} VV
@VV\delta_1 V \\ \Omega^{\even}_{\BU}(X) @>>s_k> \Omega^{2k}_\Z(X)}
\cdd{\wh{K}_{\SSSS}(X) @>\wh{c}_k>> \wh{H}^{2k}(X; \R/\Z) \\ @V\delta VV @VV\delta_2 V
\\ K(X) @>>c_k> H^{2k}(X; \Z)}
where $\ch_{\wh{K}_{\SSSS}}(\E)=\ch(\nabla)$, and $\delta(\E)=[E]$.
\end{coro}

We now prove the product formula of differential Chern classes.

\begin{prop}
Let $X$ be compact. The following diagram commutes
\begin{equation}\label{eq 3.21}
\begin{CD}
\wh{K}_{\SSSS}(X)\times\wh{K}_{\SSSS}(X) @>\oplus>> \wh{K}_{\SSSS}(X) \\ @V\wh{c}V\wh{c}V
@VV\wh{c}V \\ \wh{H}^{\even}(X; \R/\Z)\times\wh{H}^{\even}(X; \R/\Z) @>>\ast> \wh{H}^{\even}
(X; \R/\Z)
\end{CD}
\end{equation}
where $\ast$ is the product of differential characters \cite{CS85}.
\end{prop}
\begin{proof}
For a pair $(\E, \F)\in\wh{K}_{\SSSS}(X)\times\wh{K}_{\SSSS}(X)$, Remark \ref{remark 1},
the uniqueness of the ring structure of $\wh{H}^{\even}(X; \R/\Z)$, and (3) of
\cite[Theorem 1.11]{CS85} imply that $\wh{c}(\E)\ast\wh{c}(\F)$ is the unique natural
differential character whose curvature and characteristic class are given by
\begin{equation}\label{eq 3.22}
\delta_1(\wh{c}(\E)\ast\wh{c}(\F))=c(\nabla^E)\wedge c(\nabla^F)\textrm{ and }\delta_2
(\wh{c}(\E)\ast\wh{c}(\F))=c(E)\cup c(F).
\end{equation}
On the other hand, for a pair $(\E, \F)\in\wh{K}_{\SSSS}(X)\times\wh{K}_{\SSSS}(X)$,
$\wh{c}(\E\oplus\F)$ is also the unique natural differential character whose curvature
and characteristic class are given by (\ref{eq 3.22}) respectively. Thus (\ref{eq 3.21})
holds.
\end{proof}

\subsection{Differential Chern classes on $\wh{K}_{\FL}$}\label{3.3}

In this subsection we give an explicit formula of differential Chern classes on Freed-Lott
differential $K$-theory. We refer the details of Freed-Lott differential $K$-theory to
\cite{FL10}.

Defining differential Chern classes on Freed-Lott differential $K$-theory $\wh{K}_{\FL}$
involves one more issue: since generators of Freed-Lott differential $K$-group are of the
form $(E, h, \nabla, \phi)$, where \dis{\phi\in\frac{\Omega^{\odd}(X)}{d\Omega^{\even}(X)}},
we have to consider $\phi$ when defining $\wh{c}_k$ on generators of $\wh{K}_{\FL}(X)$. A
natural choice for the form part of the definition of $\wh{c}_k(E, h, \nabla, \phi)$
would be $i_2(\phi_{[2k-1]})$, where $\phi_{[2k-1]}$ is the degree $2k-1$ component of
$\phi$, as the differential Chern character $\wh{\ch}_{\FL}:\wh{K}_{\FL}(X)\to\wh{H}^{\even}
(X; \R/\Q)$ is defined in this way \cite[\S8.13]{FL10}. However, as we will see below, this
definition is not correct.

On the other hand, Simons-Sullivan differential $K$-theory is isomorphic to Freed-Lott
differential $K$-theory via unique ring isomorphisms $f:\wh{K}_{\SSSS}(X)\to\wh{K}_{\FL}(X)$
and $g:\wh{K}_{\FL}(X)\to\wh{K}_{\SSSS}(X)$ (see \cite[Theorem 1]{H12} for the definitions
of $f$ and $g$). We might define differential Chern classes on Freed-Lott differential
$K$-theory, denoted by $\wh{c}_k^{\FL}:\wh{K}_{\FL}(X)\to\wh{H}^{2k}(X; \R/\Z)$, by
$$\wh{c}_k^{\FL}(E, h^E, \nabla^E, \phi):=(\wh{c}_k\circ g)(E, h^E, \nabla^E, \phi).$$
Since the formula of $g$ is complicated, we refrain from doing so. Instead, we define
differential Chern classes on $\wh{K}_{\FL}$ directly, as follows.

\begin{prop}\label{prop 5}
Let $X$ be compact. The map $\wh{c}_k^{\FL}:\wh{K}_{\FL}(X)\to\wh{H}^{2k}(X; \R/\Z)$ defined
by
\begin{equation}\label{eq 3.23}
\wh{c}_k^{\FL}(E, h^E, \nabla^E, \phi^E)(z)=\int_M\alpha+\int_{a(z)}s_k(\ch(\nabla^E)+d
\phi^E)\mod\Z,
\end{equation}
where $z\in Z_{2k-1}(X)$, $M$, $\alpha$ and $a(z)$ are chosen as in the proof of Proposition
\ref{prop 1}, is well defined.
\end{prop}
\begin{proof}
One can prove the theorem along the lines of $S_{P, u}$. Namely, we first assume the
existence of $\wh{c}_k^{\FL}(E, h^E, \nabla^E, \phi^E)$ as a differential character with
its naturality and the compatibility with curvature to prove its uniqueness as in Proposition
\ref{prop 1}. We then get the explicit formula (\ref{eq 3.23}). Then we prove (\ref{eq 3.23})
is independent of the choices as in Proposition \ref{prop 2}, and it defines a differential
character. Then we prove the naturality and the compatibility with curvature
and with characteristic class of $\wh{c}_k^{\FL}(E, h, \nabla, \phi)$. This will imply the
uniqueness of $\wh{c}_k^{\FL}(E, h^E, \nabla^E, \phi^E)$ by Remark \ref{remark 1}.

To prove the map $\wh{c}^{\FL}_k:\wh{K}_{\FL}(X)\to\wh{H}^{2k}(X; \R/\Z)$ is well defined,
let $\E$ be a generator of $\wh{K}_{\SSSS}(X)$. By (\ref{eq 3.12}) and (\ref{eq 3.23}),
we have
$$\wh{c}_k(\E)=(\wh{c}^{\FL}_k\circ f)(\E),$$
where $f$ is given by \cite[Theorem 1]{H12}. Since $f$ and $\wh{c}_k$ are well defined,
it follows that $\wh{c}^{\FL}_k$ is well defined.
\end{proof}

In particular, analogous statements of Corollary \ref{coro 2} holds for $\wh{c}^{\FL}_k$.
From (\ref{eq 3.23}) we have $\delta_1(\wh{c}^{\FL}_k(\E))=s_k(\ch(\nabla^E)+d\phi^E)$. The
reason of choosing this term is the following. Recall that there are ring homomorphisms,
given by $\ch_{\wh{K}_{\SSSS}}(E, h, [\nabla])=\ch(\nabla)$ and $\ch_{\wh{K}_{\FL}}(E, h,
\nabla, \phi)=\ch(\nabla)+d\phi$ such that the following diagram commutes
$$\xymatrix{\wh{K}_{\SSSS}(X) \ar[r]^f \ar[dr]_{\ch_{\wh{K}_{\SSSS}}} & \wh{K}_{\FL}(X)
\ar[d]^{\ch_{\wh{K}_{\FL}}} \\ & \Omega^{\even}_{\BU}(X)}$$
where $f$ is the ring isomorphism given in \cite[Theorem 1]{H12}. Since
$$\delta_1(\wh{c}_k(E, h, [\nabla]))=c_k(\nabla)=s_k(\ch(\nabla))=s_k(\ch_{\wh{K}_{\SSSS}}
(X)),$$
it follows that we must define $\wh{c}_k^{\FL}$ so that its curvature is given by
$$s_k(\ch_{\wh{K}_{\FL}}(E, h, \nabla, \phi))=s_k(\ch(\nabla)+d\phi),$$
which implies the compatibility between differential Chern classes on Simons-Sullivan's
and Freed-Lott's models of differential $K$-theory. Thus (\ref{eq 3.23}) gives the correct
formula for differential Chern classes on the differential $K$-group defined by vector
bundles with connections and odd forms.

\subsection{Odd differential Chern classes}\label{3.4}

One can define odd differential Chern classes in a model-independent way as in
\cite[Theorem 1.2]{B09}, for which we recall here. The $(2k+1)$-th odd differential Chern
class $\wh{c}^{\odd}_{2k+1}:\wh{K}^{-1}(X)\to\wh{H}^{2k+1}(X; \R/\Z)$ is defined to be the
composition in the following diagram
\begin{equation}\label{eq 3.24}
\xymatrix{\wh{K}^{-1}(X) \ar@{.>}[rr]^{\wh{c}^{\odd}_{2k+1}} \ar[d]_S & &\wh{H}^{2k+1}(X;
\R/\Z) \\ \wh{K}(\sss^1\times X) \ar[rr]_{\wh{c}_{k+1}} & & \wh{H}^{2k+2}(\sss^1\times X;
\R/\Z) \ar[u]_{\int^{\wh{H}}_{\sss^1}}}
\end{equation}
i.e., \dis{\wh{c}^{\odd}_{2k+1}:=\int^{\wh{H}}_{\sss^1}\circ\wh{c}_{k+1}\circ S}, where
$S:\wh{K}^{-1}(X)\to\wh{K}(\sss^1\times X)$ is the suspension map and
\dis{\int^{\wh{H}}_{\sss^1}} is the integration along the fibers of $\sss^1\times X\to X$
in $\wh{H}$. The odd differential Chern classes so defined satisfy the following
commutative diagram
\cite[Theorem 1.2]{B09}
\begin{equation}\label{eq 3.25}
\begin{CD}
\wh{K}^{-1}(X) @>\wh{c}^{\odd}_{2k+1}>> \wh{H}^{2k+1}(X; \R/\Z) \\ @A\int^{\wh{K}}_{\sss^1}AA
@AA\int^{\wh{H}}_{\sss^1}A \\ \wh{K}(\sss^1\times X) @>>\wh{c}_{k+1}> \wh{H}^{2k+2}(\sss^1
\times X; \R/\Z)
\end{CD}
\end{equation}
where \dis{\int^{\wh{K}}_{\sss^1}:\wh{K}(\sss^1\times X)\to\wh{K}^{-1}(X)} is the
integration along the fibers of $\sss^1\times X\to X$ in $\wh{K}$. The proof of
(\ref{eq 3.25}) follows immediately from the definition of the integration map (see
\cite[Proposition 4.2]{BS10} and also the Appendix \ref{app A}). As in the proof of
\cite[Theorem 1.2]{B10}, $\wh{c}^{\odd}_{2k+1}$ is unique as \dis{\int^{\wh{K}}_{\sss^1}}
is surjective.

There are various models of odd differential $K$-theory \cite{BS09, FL10, TWZ13, PMSV13}.
For example, if we use the odd differential $K$-group defined in \cite{TWZ13}, which is
the odd counterpart of Simons-Sullivan differential $K$-theory, then the odd differential
Chern class defined by (\ref{eq 3.24}) is well defined by Proposition \ref{prop 3}. Denote
by $c^{\odd}_{2k+1}([g])$ the $(2k+1)$-th odd Chern class. Note that
$\delta_2(\wh{c}^{\odd}_{2k+1}([g]))=c^{\odd}_{2k+1}([g])$ by the compatibility between
\dis{\int^{\wh{H}}_{\sss^1}} with $\delta_2$ \cite[Proposition 4.2]{BS10} and the definition
of $c^{\odd}_{2k+1}([g])$.

\bibliographystyle{amsplain}
\bibliography{MBib}

\appendix
\section{A proof of (\ref{eq 3.25})}\label{app A}

For the convenience of readers we include a proof of (\ref{eq 3.25}). Note that for any
$x\in\wh{K}^{-1}(X)$, $S(x)\in\wh{K}(\sss^1\times X)$ is equal to multiplying $x$ by a
certain element in $\wh{K}^{-1}(\sss^1)$, for which we denote it by $e$.

Denote by $\wh{E}$ the differential extension of a generalized cohomology theory $E$ which
is multiplicative. Let $i:X\to\sss^1\times X$ be the inclusion map and $p:\sss^1\times X
\to X$ the projection map. Since $p\circ i=\id_X$, it follows that
$$\wh{E}(\sss^1\times X)=\im(p^*)\oplus\ker(i^*).$$
As in \cite[Proposition 4.2]{BS10}, every $x\in\ker(i^*)$ can be uniquely written as
$x=e\times y+a(\rho)$, where $y\in\wh{E}^{-1}(X)$ and \dis{\rho\in\frac{\Omega^{\odd}
(\sss^1\times X)}{\im(d)}}. Thus every $u\in\wh{E}(\sss^1\times X)$ can be written as
$$u=p^*z\oplus x=p^*z\oplus(e\times y+a(\rho)).$$
The map \dis{\int^{\wh{E}}_{\sss^1}:\wh{E}(\sss^1\times X)\to\wh{E}^{-1}(X)} is defined to
be the composition
\cdd{\wh{E}(\sss^1\times X) @>>> \ker(i^*) @>>> \wh{E}^{-1}(X)}
where the first map is the projection map. Obviously the integration map satisfies
\dis{\int^{\wh{E}}_{\sss^1}\circ p^*=0}, and is defined by
$$\int^{\wh{E}}_{\sss^1}u=y+a\bigg(\int_{\sss^1}\rho\bigg).$$
Note that
\begin{displaymath}
\begin{split}
\int^{\wh{H}}_{\sss^1}\wh{c}_{k+1}(u)&=\int^{\wh{H}}_{\sss^1}\wh{c}_{k+1}(p^*z\oplus(e\times
y+a(\rho)))\\
&=\int^{\wh{H}}_{\sss^1}p^*\wh{c}_{k+1}(z)+\int^{\wh{H}}_{\sss^1}\wh{c}_{k+1}(e\times y+a
(\rho))\\
&=\int^{\wh{H}}_{\sss^1}\wh{c}_{k+1}(e\times y+a(\rho)),
\end{split}
\end{displaymath}
and by (\ref{eq 3.24}), we have
\begin{displaymath}
\begin{split}
\wh{c}^{\odd}_{2k+1}\bigg(\int^{\wh{K}}_{\sss^1}u\bigg)&=\wh{c}^{\odd}_{2k+1}\bigg(y+a\bigg
(\int_{\sss^1}\rho\bigg)\bigg)\\
&=\int^{\wh{H}}_{\sss^1}\wh{c}_{k+1}\bigg(S\bigg(y+a\bigg(\int_{\sss^1}\rho\bigg)\bigg)\bigg)\\
&=\int^{\wh{H}}_{\sss^1}\wh{c}_{k+1}(e\times y+a(\rho)).
\end{split}
\end{displaymath}
Thus (\ref{eq 3.25}) holds.
\end{document}